\documentclass[11pt]{article}
\usepackage[a4paper,margin=1.12in]{geometry}
\usepackage{amsmath,amsthm,amssymb}
\usepackage{microtype}
\usepackage{cite}
\usepackage{hyperref}
\usepackage[nameinlink,capitalize,noabbrev]{cleveref}
\usepackage{indentfirst}

\hypersetup{
  colorlinks=true,
  linkcolor=black,
  citecolor=black,
  urlcolor=black
}

\newcommand{\CC}{\mathbb C}
\newcommand{\PP}{\mathbb P}
\DeclareMathOperator{\Gr}{Gr}

\newtheorem{theorem}{Theorem}[section]
\newtheorem{lemma}[theorem]{Lemma}
\newtheorem{proposition}[theorem]{Proposition}
\newtheorem{corollary}[theorem]{Corollary}
\theoremstyle{remark}

\title{The Aron--Rueda zero-subspace problem}
\author{Nacib Albuquerque\thanks{Department of Mathematics, Universidade Federal da Para\'iba, 58051-900 Jo\~ao Pessoa, Brazil. E-mail: \texttt{ngalbuquerque@mat.ufpb.br}}
\and
Daniel M. Pellegrino\thanks{Department of Mathematics, Universidade Federal da Para\'iba, 58051-900 Jo\~ao Pessoa, Brazil. E-mail: \texttt{daniel.pellegrino@academico.ufpb.br}. Corresponding author.}
\and
Anselmo Raposo Jr.\thanks{Coordination of the Bachelor's Degree in Mathematics, Universidade Federal do Maranh\~ao, 65085-580 S\~ao Lu\'is, Brazil. E-mail: \texttt{anselmo.junior@ufma.br}.}}
\date{}

\begin{document}
\maketitle

\begin{abstract}
We determine the exact finite-dimensional threshold in the zero-subspace problem of Aron and Rueda for
complex homogeneous polynomials.  More precisely, for every $d$ and $k$ we
determine the least $m$ such that every $d$-homogeneous polynomial on
$\CC^m$ vanishes on a $k$-dimensional linear subspace.  We also determine
the exact threshold for arbitrary polynomials of degree at most $d$ to be
constant on a $k$-dimensional linear subspace.  The two thresholds are
different. In the homogeneous case the exact threshold follows from Tevelev's
theorem on isotropic subspaces and closedness of the incidence locus. In the
bounded-degree case we first eliminate the linear homogeneous component by
passing to its kernel; the remaining components, of degrees $2,\ldots,d$,
form the system to which the Debarre--Manivel theorem is applied. For $k=2$
we give a separate proof using top Chern classes and Newton's inequalities.
\end{abstract}

\medskip
\noindent\textbf{Keywords:} homogeneous polynomial; zero subspace;
Grassmannian; isotropic subspace; symmetric form; Chern class.

\smallskip
\noindent\textbf{2020 Mathematics Subject Classification:} Primary 46G25;
Secondary 14M10, 14M15, 14N05.

\section{The problem and the answer}

Aron and Rueda asked for quantitative information on linear subspaces
contained in the zero set of a complex homogeneous polynomial
\cite{AronRueda1997}.  The problem was later considered from several related
points of view; see
\cite{AronGonzaloZagorodnyuk2000,AronHajek2006,AvilesTodorcevic2009,
AronSeoane2025,AiresBotelho2026}.  For the wider lineability background we
refer to \cite{AronGurariySeoane2005,BernalPellegrinoSeoane2014,LineabilityBook}.

The question we consider is the following.  Let
$P:\CC^m\to\CC$ be a $d$-homogeneous polynomial.  How large must $m$ be in
order to guarantee, for every such $P$, the existence of a $k$-dimensional
linear subspace on which $P$ vanishes identically?  We denote the least such
integer by $h(d,k)$.  Aron and Rueda proved that such a dimension exists and
gave quantitative bounds.  Louren\c{c}o and Tocha later obtained
substantially sharper estimates \cite{LourencoTocha2007}.  These arguments
are constructive and proceed by producing zero directions successively.
They do not, however, identify the optimal threshold.  To the best of our
knowledge, neither the exact value obtained below nor its connection with
Tevelev's isotropic-subspace theorem has been recorded in the zero-subspace
literature.

Let $\Gr(k,m)$ denote the Grassmannian of $k$-dimensional linear subspaces of
$\CC^m$. A subspace $E$ is therefore a point of $\Gr(k,m)$, while the
condition $P|_E=0$ imposes $\binom{d+k-1}{k-1}$ independent conditions. The
same parameter count appears in Tevelev's theorem for isotropic subspaces of
symmetric forms. Closedness of the incidence locus then passes from a general
form to every form.

For a nonzero complex vector space $E$, we use $\mathbb P(E)=(E\setminus\{0\})/\mathbb C^\times$; thus $\dim E=k$ corresponds to $\mathbb P(E)\simeq\mathbb P^{k-1}$.

Concretely, a $k$-dimensional vector subspace $E\subset\CC^m$ determines a
projective $(k-1)$-plane
$\PP(E)\subset\PP^{m-1}$. For a homogeneous polynomial $P$, write
\[
        V(P)=\{[x]\in\PP^{m-1}:P(x)=0\}.
\]
Then
\[
        P|_E=0
        \quad\Longleftrightarrow\quad
        \PP(E)\subset V(P).
\]
Thus the problem is equivalent to asking when every degree-$d$ hypersurface
in $\PP^{m-1}$ contains a projective $(k-1)$-plane.

The homogeneous threshold is the following.

\begin{theorem}[Homogeneous case]\label{thm:hom-main}
For all positive integers $d$ and $k$,
\[
h(d,k)=
\begin{cases}
k+1, & d=1,\\[2mm]
2k, & d=2,\\[2mm]
\displaystyle
k+\left\lceil
\frac{1}{k}\binom{d+k-1}{k-1}
\right\rceil, & d\ge3.
\end{cases}
\]
\end{theorem}

The size of the improvement is easy to see for cubic polynomials.
Louren\c{c}o and Tocha proved
\[
        h(3,k)\le
        2^k+\frac{k^2+k}{2}-1.
\]
They record the bounds
\[
        13,\ 25,\ 46
\]
for $k=3,4,5$, respectively; the earlier Aron--Rueda bounds were
$16,40,96$. The exact values given by \cref{thm:hom-main} are
\[
        h(3,3)=7,\qquad h(3,4)=9,\qquad h(3,5)=12.
\]
In fact,
\[
        h(3,k)=
        k+\left\lceil\frac{(k+1)(k+2)}6\right\rceil.
\]

We also consider the least integer $\mu(d,k)$ such that every polynomial

\[
        P:\CC^m\to\CC,
        \qquad \deg P\le d,
\]
is constant on some $k$-dimensional linear subspace. Here and throughout,
all subspaces are linear rather than affine; consequently, the constant value
of $P|_E$ is necessarily $P(0)$.

This is not the same problem. If
\[
        P=P_0+P_1+\cdots+P_d
\]
is the homogeneous decomposition of $P$, then
\[
        P|_E\text{ is constant}
        \quad\Longleftrightarrow\quad
        P_j|_E=0
        \quad (j=1,\ldots,d).
\]
Thus the homogeneous problem asks for a zero subspace of a single form,
whereas the bounded-degree problem asks for a common zero subspace of all
nonconstant homogeneous components simultaneously.

The exact answer is again explicit.

\begin{theorem}[Polynomials of degree at most $d$]\label{thm:general-main}
For all positive integers $d$ and $k$,
\[
\mu(d,k)=
\begin{cases}
k+1, & d=1,\\[2mm]
2k+1, & d=2,\\[2mm]
\displaystyle
k+\left\lceil
\frac{\binom{d+k}{k}-1}{k}
\right\rceil, & d\ge3.
\end{cases}
\]
\end{theorem}

In particular,
\[
        h(2,k)=2k,
        \qquad
        \mu(2,k)=2k+1.
\]
So even in degree two the homogeneous and nonhomogeneous thresholds are
different.  The quadratic case is exceptional and is governed by the Witt theory of
quadratic forms, so we treat it separately.

The homogeneous problem is linked to isotropic subspaces of symmetric forms,
whereas the bounded-degree problem becomes a simultaneous system of degrees
$(1,2,\ldots,d)$. The relevant geometric facts are recalled next.

\section{The geometric reduction}
\label{sec:ag-guide}

For a fixed $k$-dimensional subspace $E$, vanishing gives linear conditions
on the coefficients, while $E$ ranges over the Grassmannian of dimension
$k(m-k)$.  This gives the candidate threshold and the sharp converse.
Existence at the threshold comes from Tevelev in the homogeneous case and
from Debarre--Manivel in the bounded-degree case.

\subsection{Projectivization}

Let $E\subset\CC^m$ be a $k$-dimensional linear subspace. As defined above,
$\PP(E)$ is a projective space of dimension $k-1$, naturally embedded as a
projective $(k-1)$-plane in $\PP^{m-1}$.  Conversely, every projective
$(k-1)$-plane in $\PP^{m-1}$ arises uniquely in this way from a
$k$-dimensional vector subspace of $\CC^m$.

If
\[
        G:\CC^m\longrightarrow\CC
\]
is homogeneous of positive degree, then
\begin{equation}\label{eq:proj-dictionary}
        G|_E=0
        \quad\Longleftrightarrow\quad
        \PP(E)\subset V(G)\subset\PP^{m-1}.
\end{equation}
Indeed, homogeneity makes the zero set invariant under multiplication by
nonzero scalars.  Thus vanishing on all of $E$ is equivalent to vanishing on
the corresponding projective points.

For a system of homogeneous forms $G_1,\ldots,G_s$, the same argument gives
\[
        G_i|_E=0\quad (i=1,\ldots,s)
\]
if and only if
\[
        \PP(E)\subset V(G_1,\ldots,G_s).
\]
No smoothness or complete-intersection hypothesis is involved.

\subsection{The Grassmannian}

The Grassmannian $\Gr(k,m)$ is a projective algebraic variety.
For the present paper, the only numerical fact about it that we need is
\begin{equation}\label{eq:grass-dim}
        \dim\Gr(k,m)=k(m-k).
\end{equation}

Near a fixed decomposition
\[
        \CC^m=E\oplus F,
        \qquad \dim E=k,\quad \dim F=m-k,
\]
every subspace sufficiently close to $E$ is the graph of a linear map
$E\to F$.  Such maps form a vector space of dimension $k(m-k)$, which is the
local dimension of the Grassmannian.

Thus \eqref{eq:grass-dim} is the number of parameters available for the
choice of the zero subspace.

\subsection{The number of conditions}

Fix a $k$-dimensional subspace $E\subset\CC^m$ and a positive integer $e$.
A homogeneous polynomial of degree $e$ vanishes identically on $E$ exactly
when its restriction to $E$ is the zero polynomial.  The vector space of
degree-$e$ homogeneous polynomials on $E\simeq\CC^k$ has dimension
\begin{equation}\label{eq:number-monomials}
        \binom{e+k-1}{k-1}.
\end{equation}
It is the number of monomials
\[
        z_1^{\alpha_1}\cdots z_k^{\alpha_k},
        \qquad
        \alpha_1+\cdots+\alpha_k=e.
\]

Restriction from $\CC^m$ to $E$ is surjective.  Indeed, after choosing
coordinates with
\[
        E=\{z_{k+1}=\cdots=z_m=0\},
\]
a homogeneous polynomial in $z_1,\ldots,z_k$ is the restriction of the same
polynomial viewed as a polynomial in all $m$ variables.  Therefore the
condition
\[
        G|_E=0
\]
imposes exactly
\[
        \binom{e+k-1}{k-1}
\]
independent linear conditions on a degree-$e$ form.

For forms of degrees $e_1,\ldots,e_s$, the total number of conditions at a
fixed $E$ is consequently
\begin{equation}\label{eq:cdef-guide}
        c(\mathbf e,k)
        :=
        \sum_{i=1}^s\binom{e_i+k-1}{k-1}.
\end{equation}

Comparing \eqref{eq:grass-dim} and \eqref{eq:cdef-guide} already suggests
the answer.  There are $k(m-k)$ parameters with which to choose $E$, and
there are $c(\mathbf e,k)$ conditions to satisfy.  Hence the critical sign is
\begin{equation}\label{eq:expected-vector}
        k(m-k)-c(\mathbf e,k).
\end{equation}
For a single form the same count is exactly the one occurring in Tevelev's
isotropic-subspace theorem.  For systems, we shall use the projective
nonemptiness theorem recalled next.

\subsection{The Debarre--Manivel theorem}

We shall use a small amount of standard algebraic-geometric language.  The space
\[
        H^0(\PP^N,\mathcal O(e))
\]
is simply the vector space of homogeneous polynomials of degree $e$ in $N+1$
variables.  A condition on the coefficients is called \emph{Zariski closed} if it
is given by polynomial equations in those coefficients; a Zariski-open set is the
complement of a Zariski-closed set.  We say that a property holds for a
\emph{general} form, or a general system of forms, when it holds on some nonempty
Zariski-open subset of the corresponding parameter space.  The projective spaces
used below are irreducible, so every nonempty Zariski-open subset is dense.

For the bounded-degree problem we use the following result of
Debarre--Manivel.

Let $\mathbf e=(e_1,\ldots,e_s)$ be a finite sequence of positive integers, and let
\[
        X=V(G_1,\ldots,G_s)\subset\PP^N,
        \qquad \deg G_i=e_i.
\]
We write $F_r(X)$ for the parameter space of projective $r$-planes contained in $X$.  The Grassmannian of all
projective $r$-planes in $\PP^N$ has dimension
\[
        (r+1)(N-r),
\]
and the restriction to a fixed $r$-plane of a degree-$e_i$ equation has
\[
        \binom{e_i+r}{r}
\]
coefficients.  The difference between these two numbers is the expected dimension. Thus
\begin{equation}\label{eq:delta}
\delta(N,\mathbf e,r)
=
(r+1)(N-r)
-
\sum_{i=1}^s\binom{e_i+r}{r}.
\end{equation}

Debarre and Manivel introduce the auxiliary integer
\begin{equation}\label{eq:delta-minus}
\delta_-(N,\mathbf e,r)
=
\min\bigl\{\delta(N,\mathbf e,r),\,N-2r-s\bigr\}.
\end{equation}
The precise statement we need is contained in
\cite[Theorem~2.1(b)]{DebarreManivel1998}.

\begin{theorem}[Debarre--Manivel, Theorem~2.1(b), in our notation]
\label{thm:dm-fano}
Let $e_1,\ldots,e_s\ge2$, and let $X\subset\PP^N$ be defined by a general
system of homogeneous equations of multidegree
$\mathbf e=(e_1,\ldots,e_s)$. If
\[
        \delta_-(N,\mathbf e,r)\ge0,
\]
then $F_r(X)$ is nonempty and smooth of dimension $\delta(N,\mathbf e,r)$.
\end{theorem}

In the bounded-degree application we check
\[
        \delta(N,\mathbf e,r)\ge0,
        \qquad
        N-2r-s\ge0
\]
separately.

Bastianelli--Ciliberto--Flamini--Supino
\cite{BastianelliEtAl2020} use the opposite sign convention
\[
 t=
 \sum_{i=1}^s\binom{e_i+r}{r}-(r+1)(N-r)
 =-\delta.
\]
We use the $\delta_-$ formulation above.

\subsection{Closed incidence}

Fix positive integers $e_1,\ldots,e_s$, $k$ and $m$ with $1\le k<m$, and set
\[
\mathcal P
=
\prod_{i=1}^s
\PP\!\left(H^0(\PP^{m-1},\mathcal O(e_i))\right).
\]
Here $H^0(\PP^{m-1},\mathcal O(e_i))$ is, concretely, the space of
degree-$e_i$ homogeneous polynomials on $\CC^m$.  Thus a point of $\mathcal P$
is a tuple of nonzero homogeneous forms of the prescribed degrees, where each
form is considered up to multiplication by a nonzero scalar.  This quotient is harmless because multiplying a form by a
nonzero scalar does not change its zero set.
For the sharpness argument it is convenient to projectivize the factors
separately.  This differs harmlessly from the projectivization of the direct
sum used in some treatments of the universal family: independent rescaling of
the equations does not change their common zero locus, while the product
notation makes the codimension count completely transparent.

Define
\[
\mathcal I=
\left\{
\bigl(([G_1],\ldots,[G_s]),E\bigr)
\in
\mathcal P\times\Gr(k,m):
G_i|_E=0\ \text{for every }i
\right\}.
\]
Thus $\mathcal I$ remembers simultaneously a system of forms and a common
$k$-dimensional zero subspace.  Let
\[
        \pi:\mathcal I\longrightarrow\mathcal P
\]
be the projection and write
\[
        \mathcal W:=\pi(\mathcal I).
\]
Then $\mathcal W$ is exactly the set of systems admitting at least one common
$k$-dimensional zero subspace.

\begin{lemma}[Closed-incidence principle]\label{lem:closure}
The incidence set $\mathcal I$ is Zariski closed in
$\mathcal P\times\Gr(k,m)$, and its image
$\mathcal W=\pi(\mathcal I)$ is Zariski closed in $\mathcal P$.
\end{lemma}

\begin{proof}
We prove both assertions directly.

\smallskip
\noindent\emph{Step 1: the condition $G_i|_E=0$ is algebraic.}
On an affine coordinate chart of the Grassmannian, a $k$-dimensional
subspace $E$ can be represented as the span of $k$ vectors whose coordinates
depend polynomially on the chart parameters.  Write a general vector of
$E$ as
\[
        x_1v_1+\cdots+x_kv_k.
\]
After substitution into $G_i$, we obtain a homogeneous polynomial of degree
$e_i$ in the variables $x_1,\ldots,x_k$.  The restriction $G_i|_E$ is zero
exactly when every coefficient of this polynomial is zero.  Those
coefficients are polynomial expressions in the coefficients of $G_i$ and in
the coordinates describing $E$.  The same description holds on the standard
affine charts of the projective parameter factors, and these charts together
with the Grassmannian charts cover $\mathcal P\times\Gr(k,m)$.  Thus the
intersection of $\mathcal I$ with every chart is Zariski closed, and hence
$\mathcal I$ itself is Zariski closed.

\smallskip
\noindent\emph{Step 2: forgetting $E$ preserves closedness.}
Here we use the standard projective closed-map theorem: if $Y$ is projective,
then the projection $X\times Y\to X$ maps Zariski-closed sets to
Zariski-closed sets.  Since $\Gr(k,m)$ is projective, applying this to
\[
        \mathcal P\times\Gr(k,m)\longrightarrow\mathcal P
\]
shows that
\[
        \mathcal W=\pi(\mathcal I)
\]
is Zariski closed.

\end{proof}

\subsection{Tevelev and the homogeneous case}

Let $P$ be a $d$-homogeneous polynomial and let $A$ be its associated
symmetric $d$-linear form, so that $P(x)=A(x,\ldots,x)$.  By polarization,
for every linear subspace $E$ one has
\[
        P|_E=0
        \quad\Longleftrightarrow\quad
        A|_{E^d}=0.
\]
Thus the zero subspaces in the sense of Aron--Rueda are exactly the isotropic
subspaces in Tevelev's terminology.  For a symmetric $d$-linear form a
subspace $E$ is called isotropic when its restriction to $E^d$ vanishes
identically.  Tevelev determined the maximal isotropic
dimension for forms in general position.  We need only the symmetric case.

\begin{theorem}[Tevelev, symmetric case]\label{thm:tevelev}
Let $V$ be a complex vector space of dimension $m$, let $d\ge3$, and let
$k\ge1$.  A symmetric $d$-linear form in general position on $V$ has a
$k$-dimensional isotropic subspace if and only if
\[
        m\ge
        k+\frac1k\binom{d+k-1}{d}.
\]
Equivalently, since $m$ is an integer, the least such dimension is
\[
        k+\left\lceil
        \frac1k\binom{d+k-1}{k-1}
        \right\rceil.
\]
\end{theorem}

The theorem is Tevelev's result on isotropic subspaces of symmetric
polylinear forms \cite[Theorem~1]{Tevelev2001}.  For $d=2$, the universal
threshold is $m\ge2k$; this follows from \cref{lem:quadratic-witt} below.

\begin{corollary}[Generic existence becomes universal]\label{cor:tevelev-universal}
Let $d\ge3$ and set
\[
 m_0=
 k+\left\lceil
 \frac1k\binom{d+k-1}{k-1}
 \right\rceil.
\]
Then every $d$-homogeneous polynomial on $\CC^{m_0}$ vanishes on a
$k$-dimensional linear subspace.
\end{corollary}

\begin{proof}
Let
\[
\mathcal P_d=\PP\!\left(H^0(\PP^{m_0-1},\mathcal O(d))\right),
\]
the projective space of nonzero $d$-homogeneous polynomials on $\CC^{m_0}$,
and let $\mathcal W_d\subset
\mathcal P_d$ be the subset consisting of forms admitting a
$k$-dimensional zero subspace.  It is the image of the incidence
correspondence considered in \cref{lem:closure}; hence $\mathcal W_d$ is
Zariski closed.  By \cref{thm:tevelev}, $\mathcal W_d$ contains the
nonempty Zariski-open set of forms in general position.  Since
$\mathcal P_d$ is irreducible, this open set is dense.  A closed subset
containing a dense subset is the entire space, so
$\mathcal W_d=\mathcal P_d$.  The zero polynomial is immediate.
\end{proof}

\subsection{The dimension count}

The opposite direction requires no generic Fano theorem.

\begin{lemma}[Dimension of the incidence correspondence]
\label{lem:incidence-dimension}
With the notation above,
\[
\dim\mathcal I
=
\dim\mathcal P+k(m-k)
-
\sum_{i=1}^s\binom{e_i+k-1}{k-1}.
\]
Hence, if
\begin{equation}\label{eq:negative-incidence}
k(m-k)
<
\sum_{i=1}^s\binom{e_i+k-1}{k-1},
\end{equation}
then $\mathcal W$ is a proper closed subset of $\mathcal P$.  In particular,
a general system has no common $k$-dimensional zero subspace.
\end{lemma}

\begin{proof}
Fix $E\in\Gr(k,m)$.  For each $i$, restriction gives a surjective linear map
\[
H^0(\PP^{m-1},\mathcal O(e_i))
\longrightarrow
H^0(\PP(E),\mathcal O_{\PP(E)}(e_i)).
\]
As explained above, the target has dimension
\[
        \binom{e_i+k-1}{k-1}.
\]
Therefore, over a fixed $E$, the condition $G_i|_E=0$ has exactly this
codimension in the $i$th parameter factor.  The forms $G_1,\ldots,G_s$
belong to different factors of the product parameter space $\mathcal P$, so
these codimensions add.  Thus the fiber over a fixed $E$ has dimension
\[
\underbrace{\dim\mathcal P}_{\text{all systems}}
-
\underbrace{\sum_{i=1}^s\binom{e_i+k-1}{k-1}}_
            {\text{conditions for vanishing on }E}.
\]
Since the possible subspaces $E$ themselves form $\Gr(k,m)$, of dimension
$k(m-k)$, we obtain
\[
\dim\mathcal I
=
\underbrace{k(m-k)}_{\text{choice of }E}
+
\underbrace{\dim\mathcal P
-
\sum_{i=1}^s\binom{e_i+k-1}{k-1}}_
{\text{systems vanishing on that }E},
\]
which is the desired formula.

If \eqref{eq:negative-incidence} holds, then
\[
        \dim\mathcal I<\dim\mathcal P.
\]
Since the dimension of an image cannot exceed the dimension of its source,
$\mathcal W=\pi(\mathcal I)$ cannot equal $\mathcal P$.  By
\cref{lem:closure}, it is closed, hence proper.
\end{proof}

\subsection{Systems of forms}

Write
\[
        \Delta(m)
        =
        k(m-k)-c(\mathbf e,k),
        \qquad
        c(\mathbf e,k)
        =
        \sum_{i=1}^s\binom{e_i+k-1}{k-1}.
\]
The first term counts the parameters available for choosing $E$, and the
second counts the equations imposed by requiring all forms to vanish on
$E$.  The negative side is unconditional:
\[
\Delta(m)<0
\quad\Longrightarrow\quad
\dim\mathcal I<\dim\mathcal P
\quad\Longrightarrow\quad
\text{some systems have no }E.
\]
On the nonnegative side, Debarre--Manivel gives existence for a general
system when the second part of $\delta_-$, namely $N-2r-s\ge0$, also holds.
Closed incidence then promotes this generic existence to every system. Thus
the first integer for which $\Delta(m)\ge0$ is the natural candidate threshold,
and it is the exact threshold whenever that additional range inequality holds
there.
Solving
\[
        k(m-k)\ge c(\mathbf e,k)
\]
gives
\[
        m\ge k+\frac{c(\mathbf e,k)}{k},
\]
and hence the candidate threshold
\[
        m_0=
        k+\left\lceil\frac{c(\mathbf e,k)}{k}\right\rceil.
\]

This gives the following threshold principle.

\begin{proposition}[Threshold principle for systems]\label{prop:system-threshold}
Let $e_1,\ldots,e_s\ge2$ and let $k\ge2$. Set
\[
c(\mathbf e,k)
=
\sum_{i=1}^s\binom{e_i+k-1}{k-1},
\qquad
m_0
=
k+\left\lceil\frac{c(\mathbf e,k)}{k}\right\rceil.
\]
Assume
\begin{equation}\label{eq:system-safe-range}
        m_0\ge 2k+s-1.
\end{equation}
Then $m_0$ is the least integer $m$ such that every tuple of homogeneous
forms
\[
        G_i:\CC^m\longrightarrow\CC,
        \qquad \deg G_i=e_i,
        \quad i=1,\ldots,s,
\]
has a common $k$-dimensional zero subspace.
\end{proposition}

\begin{proof}
Let $m=m_0$ and put
\[
        N=m_0-1,\qquad r=k-1.
\]
Then
\[
\begin{aligned}
\delta(N,\mathbf e,r)
&=(r+1)(N-r)-\sum_{i=1}^s\binom{e_i+r}{r}\\
&=k(m_0-k)-c(\mathbf e,k)\ge0,
\end{aligned}
\]
by the definition of $m_0$. Also,
\[
N-2r-s=m_0-2k-s+1\ge0
\]
by \eqref{eq:system-safe-range}. Hence $\delta_-(N,\mathbf e,r)\ge0$.
Assume first that every $G_i$ is nonzero. Let $\mathcal P$ be the product
parameter space introduced above, and let $W\subset\mathcal P$ be the set of
systems admitting a common $k$-dimensional zero subspace. By
\cref{lem:closure}, $W$ is Zariski closed. Since
$\delta_-(N,\mathbf e,r)\ge0$, \cref{thm:dm-fano} shows that $W$
contains a nonempty Zariski-open set of general systems. The space
$\mathcal P$ is irreducible, so this open set is dense. Hence the closed set
$W$ equals $\mathcal P$. Thus every tuple of nonzero forms has a common
projective $(k-1)$-plane, equivalently a common $k$-dimensional zero subspace.
If some components are identically zero, replace each of them temporarily by
an arbitrary nonzero form of the same degree and apply the preceding argument
to the modified system. The resulting subspace also works for the original
system, since a zero component imposes no condition.

\smallskip
\noindent\emph{Sharpness.}
If $m<k$, no $k$-dimensional subspace of $\CC^m$ exists.  If $m=k$, the only
candidate is $\CC^k$ itself, and any nonzero form gives a counterexample.
Now assume
\[
        k<m<m_0.
\]
The integer $m_0-k$ is the least integer $q$ satisfying
\[
        kq\ge c(\mathbf e,k).
\]
Since $m-k<m_0-k$, it follows that
\[
        k(m-k)<c(\mathbf e,k).
\]
By \cref{lem:incidence-dimension},
\[
        \dim\mathcal I<\dim\mathcal P.
\]
Thus the incidence image cannot cover the parameter space, so there are
tuples with no common $k$-dimensional zero subspace.  Hence no smaller $m$
has the universal property.

\smallskip
\noindent\emph{Larger dimensions.}
Let $m>m_0$.  Choose any $m_0$-dimensional vector subspace
$F\subset\CC^m$ and restrict every $G_i$ to $F$.  The case $m_0$ gives a $k$-dimensional subspace
$E\subset F$ on which all restrictions vanish.  Then the original forms
also vanish on $E$.  Thus the property persists for every $m\ge m_0$.
\end{proof}

\section{Proof of the homogeneous formula}

We first isolate the elementary one-dimensional case.

\begin{lemma}\label{lem:k1hom}
For every $d\ge2$,
\[
        h(d,1)=2.
\]
\end{lemma}

\begin{proof}
One variable does not suffice, since $z\mapsto z^d$ has no nonzero zero.
Now let $P(z,w)$ be a nonzero $d$-homogeneous polynomial in two variables.
Over $\CC$, every binary homogeneous polynomial factors into homogeneous
linear factors.  Equivalently, after dehomogenizing in a chart in which the
result is nonconstant, the fundamental theorem of algebra gives a root.
Thus $P$ has a nonzero zero, and two variables always suffice.
\end{proof}

We also use the following standard consequence of the Witt
decomposition over $\CC$.

\begin{lemma}[Complex quadratic forms]\label{lem:quadratic-witt}
Let $Q$ be a quadratic form of rank $\rho$ on a complex vector space $V$ of
dimension $m$.  Then the largest possible dimension of a linear subspace
$E\subset V$ on which $Q$ vanishes identically is
\[
        (m-\rho)+\left\lfloor\frac{\rho}{2}\right\rfloor
        =
        m-\left\lceil\frac{\rho}{2}\right\rceil.
\]
Hence every quadratic form on $\CC^{2k}$ vanishes on a
$k$-dimensional subspace, while a nondegenerate quadratic form on
$\CC^{2k-1}$ has no such $k$-dimensional subspace.
\end{lemma}

\begin{proof}
Associate with $Q$ the symmetric bilinear form
\[
        B(x,y)
        =
        \frac12\bigl(Q(x+y)-Q(x)-Q(y)\bigr).
\]
Then $Q(x)=B(x,x)$.  Moreover, if $Q$ vanishes identically on a linear
subspace $E$, then for $x,y\in E$,
\[
        B(x,y)
        =
        \frac12\bigl(Q(x+y)-Q(x)-Q(y)\bigr)
        =0.
\]
Thus $Q|_E=0$ exactly when $B$ vanishes on $E\times E$.

Let
\[
        R=\{x\in V:B(x,y)=0\text{ for every }y\in V\}
\]
be the radical.  Since the rank is $\rho$,
\[
        \dim R=m-\rho.
\]
Choose a complementary subspace $U$ so that
\[
        V=R\oplus U.
\]
The restriction of $B$ to $U$ is nondegenerate and $\dim U=\rho$.

We first obtain an upper bound.  If $E$ is a subspace with
$B|_{E\times E}=0$, project $E$ onto $U$ along $R$.  The kernel of this
projection is $E\cap R$, and the image is a totally isotropic subspace of
the nondegenerate space $U$.  If $W\subset U$ is totally isotropic, then
\[
        W\subset W^\perp.
\]
Nondegeneracy gives
\[
        \dim W+\dim W^\perp=\rho,
\]
and therefore
\[
        2\dim W\le\rho.
\]
It follows that
\[
\dim E
\le
\dim R+\left\lfloor\frac{\rho}{2}\right\rfloor
=
(m-\rho)+\left\lfloor\frac{\rho}{2}\right\rfloor.
\]

For the matching lower bound, over $\CC$ a nondegenerate quadratic form can
be diagonalized.  After a linear change of coordinates on $U$ we may write
\[
        Q|_U=u_1^2+\cdots+u_\rho^2.
\]
For each $j\le\lfloor\rho/2\rfloor$, set
\[
        v_j=e_{2j-1}+i e_{2j}.
\]
Then $Q(v_j)=0$, and vectors arising from distinct pairs are mutually
orthogonal for $B$.  Hence $Q$ vanishes on
$\operatorname{span}\{v_1,\ldots,v_{\lfloor\rho/2\rfloor}\}$.  Adding the
radical $R$ gives a zero subspace of dimension
\[
        (m-\rho)+\left\lfloor\frac{\rho}{2}\right\rfloor.
\]
This matches the upper bound.
\end{proof}

\begin{proof}[Proof of \cref{thm:hom-main}]
For $d=1$, a nonzero linear functional on $\CC^m$ has kernel of dimension
$m-1$, while the zero functional causes no difficulty. Hence
\[
        h(1,k)=k+1.
\]
The case $k=1$ and $d\ge2$ is \cref{lem:k1hom}. We may therefore assume
$d\ge2$ and $k\ge2$.

Suppose first that
$d\ge3$, and set
\[
 m_0
 =
 k+\left\lceil
 \frac1k\binom{d+k-1}{k-1}
 \right\rceil.
\]
By \cref{cor:tevelev-universal}, every $d$-homogeneous polynomial on
$\CC^{m_0}$ has a $k$-dimensional zero subspace.  Thus $h(d,k)\le m_0$.
If $m<m_0$, then
\[
 k(m-k)<\binom{d+k-1}{k-1},
\]
and \cref{lem:incidence-dimension}, applied
to a single degree-$d$ form, shows that a general form has no such subspace.
Hence the universal property fails below $m_0$, and therefore
\[
        h(d,k)=m_0.
\]

It remains to consider $d=2$. By \cref{lem:quadratic-witt}, every complex
quadratic form on $\CC^{2k}$ vanishes on a $k$-dimensional subspace, whereas
a nondegenerate quadratic form on $\CC^{2k-1}$ does not. Consequently
\[
        h(2,k)=2k.
\]
\end{proof}

\section{Proof of the bounded-degree formula}

We begin with the elementary observation that separates this problem from
the homogeneous one.

\begin{lemma}\label{lem:components}
Let
\[
        P=P_0+P_1+\cdots+P_d
\]
be the homogeneous decomposition of a polynomial on $\CC^m$, and let
$E\subset\CC^m$ be linear. Then
\[
        P|_E\text{ is constant}
\]
if and only if
\[
        P_j|_E=0
        \qquad (j=1,\ldots,d).
\]
\end{lemma}

\begin{proof}
If $P|_E$ is constant, then for every $x\in E$ and $t\in\CC$,
\[
        P(tx)=P(0).
\]
Hence
\[
        \sum_{j=1}^d t^jP_j(x)=0
        \qquad\text{for every }t\in\CC.
\]
The left-hand side is a polynomial in $t$, so every coefficient is zero.
Thus $P_j(x)=0$ for $j=1,\ldots,d$. The converse is immediate.
\end{proof}

The case $k=1$ is again elementary.

\begin{lemma}\label{lem:k1general}
For every $d\ge1$,
\[
        \mu(d,1)=d+1.
\]
\end{lemma}

\begin{proof}
By \cref{lem:components}, a line $\CC v$ is a constancy subspace exactly
when
\[
        P_1(v)=\cdots=P_d(v)=0.
\]
We use the following standard algebraic fact in its simplest form:
$d$ homogeneous polynomials of positive degree in $d+1$ complex variables
have a common nonzero zero.  Here is the short reason.  Let
\[
        I=(P_1,\ldots,P_d)
        \subset
        \CC[z_0,\ldots,z_d].
\]
If the only common affine zero were the origin, the Nullstellensatz would
give
\[
        \sqrt I=(z_0,\ldots,z_d).
\]
Here $\sqrt I$ is the radical of $I$, and the ideal on the right is the
maximal ideal of the origin.  Its height is $d+1$; geometrically, the origin
has codimension $d+1$ in $\CC^{d+1}$.  On the other hand, Krull's height
theorem says that an ideal generated by only $d$ elements has height at most
$d$.  This is impossible. Hence the forms have a common nonzero zero, or
equivalently a common point of $\PP^d$.

Applying this to $P_1,\ldots,P_d$ shows that every polynomial on
$\CC^{d+1}$ of degree at most $d$ has a constancy line.  Therefore
\[
        \mu(d,1)\le d+1.
\]

For the reverse inequality, work in $\CC^d$ with coordinates
$z_1,\ldots,z_d$ and define
\[
        P(z)=\sum_{j=1}^d z_j^j.
\]
Its $j$-homogeneous component is $P_j(z)=z_j^j$, and the only common affine
zero of $P_1,\ldots,P_d$ is the origin. Thus there is no constancy line, and
$\mu(d,1)>d$.
\end{proof}

\begin{lemma}[A useful size estimate]
\label{lem:bounded-range}
Let $d\ge3$ and $k\ge2$, and put
\[
m_0
=
k+
\left\lceil
\frac{\binom{d+k}{k}-1}{k}
\right\rceil.
\]
Then
\[
        m_0\ge2k+d.
\]
Consequently,
\[
        m_0-1\ge2k+d-1>2k+d-2,
\]
which is more than the range condition needed to apply
\cref{prop:system-threshold} in dimension $m_0-1$ to the
multidegree $(2,3,\ldots,d)$.
\end{lemma}

\begin{proof}
Define
\[
        D_k(d)
        =
        \binom{d+k}{k}-1-k(k+d-1).
\]
At $d=3$,
\[
\begin{aligned}
D_k(3)
&=
\binom{k+3}{3}-1-k(k+2)\\
&=
\frac{k(k-1)(k+1)}6
>0.
\end{aligned}
\]
Moreover,
\[
\begin{aligned}
D_k(d+1)-D_k(d)
&=
\binom{d+k}{k-1}-k\\
&>0
\end{aligned}
\]
for $d\ge3$ and $k\ge2$.  Thus $D_k(d)>0$ for every $d\ge3$, so
\[
        \binom{d+k}{k}-1>k(k+d-1).
\]
After division by $k$,
\[
        \frac{\binom{d+k}{k}-1}{k}>k+d-1.
\]
The left-hand side therefore has ceiling at least $k+d$, and hence
\[
        m_0\ge k+(k+d)=2k+d.
\]
\end{proof}

\begin{proof}[Proof of \cref{thm:general-main}]
The case $k=1$ is \cref{lem:k1general}. For $d=1$, the kernel argument gives
\[
        \mu(1,k)=k+1.
\]
We may now assume $d\ge2$ and $k\ge2$.

Suppose first that $d\ge3$. By \cref{lem:components}, we seek a
$k$-dimensional subspace on which
\[
        P_1,\ldots,P_d
\]
vanish simultaneously. Set
\[
m_0
=
k+
\left\lceil
\frac{\binom{d+k}{k}-1}{k}
\right\rceil.
\]
We first remove the linear component. If $P_1\ne0$, let
$H=\ker P_1\subset\CC^{m_0}$; then $\dim H=m_0-1$. If $P_1=0$, choose any
hyperplane $H\subset\CC^{m_0}$, again with $\dim H=m_0-1$. On $H$ it remains
to find a common $k$-dimensional zero subspace for the restricted forms
\[
        P_2|_H,\ldots,P_d|_H,
\]
whose degrees are $2,3,\ldots,d$. Their total number of conditions at a fixed
$k$-plane is
\[
\begin{aligned}
 c'&=\sum_{j=2}^d\binom{j+k-1}{k-1}\\
   &=\binom{d+k}{k}-1-k.
\end{aligned}
\]
Since $\lceil x-1\rceil=\lceil x\rceil-1$,
\[
\begin{aligned}
k+\left\lceil\frac{c'}{k}\right\rceil
&=k+\left\lceil\frac{\binom{d+k}{k}-1}{k}-1\right\rceil\\
&=m_0-1.
\end{aligned}
\]
Thus $\dim H$ is exactly the candidate threshold of
\cref{prop:system-threshold} for the multidegree
$(2,3,\ldots,d)$, which has $s=d-1$. By \cref{lem:bounded-range},
\[
        m_0-1\ge2k+d-1>2k+d-2=2k+s-1,
\]
so the range condition \eqref{eq:system-safe-range} holds. The proposition
therefore gives a $k$-dimensional subspace $E\subset H$ on which
$P_2,\ldots,P_d$ vanish. Since $E\subset H$, also $P_1|_E=0$, and
\cref{lem:components} shows that $P|_E$ is constant.

For sharpness below $m_0$, use the full system of degrees
$(1,2,\ldots,d)$ only in the incidence-dimension argument. If $m<m_0$, then
\[
 k(m-k)<\sum_{j=1}^d\binom{j+k-1}{k-1}
 =\binom{d+k}{k}-1,
\]
so \cref{lem:incidence-dimension} gives a tuple
$(P_1,\ldots,P_d)$ with no common $k$-dimensional zero subspace. The
polynomial $P=P_1+\cdots+P_d$ then has no $k$-dimensional constancy subspace by
\cref{lem:components}. Hence the asserted formula for $\mu(d,k)$ is
sharp.

Finally, let $d=2$. Write
\[
        P=P_0+L+Q,
\]
where $L$ is linear and $Q$ is quadratic. If $P$ is constant on $E$, then
\[
        E\subset\ker L,
        \qquad
        Q|_E=0.
\]
If $m\ge2k+1$, then $\ker L$ has dimension at least $2k$, and
$Q|_{\ker L}$ has a totally isotropic $k$-dimensional subspace by
\cref{lem:quadratic-witt}. Hence
\[
        \mu(2,k)\le2k+1.
\]

Conversely, on $\CC^{2k}$ define
\[
L(z)=z_{2k},
\qquad
Q(z)=z_1^2+\cdots+z_{2k-1}^2.
\]
Then $Q|_{\ker L}$ is nondegenerate on the $(2k-1)$-dimensional space
$\ker L$. By \cref{lem:quadratic-witt}, its maximal totally isotropic
dimension is $k-1$. Thus no $k$-dimensional constancy subspace exists, and
\[
        \mu(2,k)=2k+1.
\]
\end{proof}

\section{The first nontrivial case: a direct proof for \texorpdfstring{$k=2$}{k=2}}
\label{sec:k2-direct}

We give a separate proof for $k=2$ which uses neither Tevelev nor
Debarre--Manivel.  The relevant top Chern class reduces to a two-variable
real-rooted polynomial, and Newton's inequalities detect the critical
Schubert class.  The same argument works in the homogeneous and bounded-degree
cases.

Let
\[
        G=\Gr(2,m)
\]
and let $S\to G$ be the tautological rank-$2$ bundle.  A
$d$-homogeneous polynomial $P$ on $\CC^m$ defines a continuous, in fact
algebraic, section
\[
        s_P:G\longrightarrow \operatorname{Sym}^d(S^*),
        \qquad
        s_P(E)=P|_E.
\]
Thus $s_P(E)=0$ exactly when $E$ is a two-dimensional zero subspace of $P$.
Consequently, if the top Chern class of $\operatorname{Sym}^d(S^*)$ is
nonzero, no section can be nowhere zero; hence every such section vanishes
somewhere.

Here is the bundle-theoretic mechanism in more detail.  The fiber of $S$ over
$E\in G$ is $S_E=E$, so the fiber of $\operatorname{Sym}^d(S^*)$ is precisely
the space of degree-$d$ homogeneous forms on $E$.  Moreover, a nowhere-zero
section of a rank-$r$ complex vector bundle determines a trivial line
subbundle and therefore forces its top Chern class $c_r$ to vanish.  Thus it
remains to show that the top Chern class of the restriction bundle is nonzero.
We use the splitting principle only as a computational device: if $x,y$ are
formal Chern roots of $S^*$, then the Chern classes of its symmetric powers
are computed as symmetric polynomials in $x,y$.  Finally, the Schur classes
indexed by partitions contained in the $2\times(m-2)$ rectangle form the
Schubert basis of $H^*(\Gr(2,m);\mathbb Z)$; this will allow us to detect a
nonzero term in the computed class.

Set
\[
        r=d+1,
        \qquad
        q=\left\lceil\frac r2\right\rceil,
        \qquad
        m=2+q.
\]
Thus $r$ is the rank of $\operatorname{Sym}^d(S^*)$.
We prove
\begin{equation}\label{eq:k2-top-chern}
        c_r(\operatorname{Sym}^d S^*)\ne0
        \quad\text{in }H^{2r}(\Gr(2,2+q);\mathbb Z).
\end{equation}

\begin{lemma}[The two-variable Euler polynomial]\label{lem:k2-euler}
Let $x,y$ denote the formal Chern roots of $S^*$.  Then
\[
F_d(x,y)
:=c_r(\operatorname{Sym}^dS^*)
=
\prod_{a=0}^d\bigl(ax+(d-a)y\bigr).
\]
If
\[
        F_d(x,y)=\sum_{j=0}^{r} C_jx^jy^{r-j},
\]
then
\[
        C_j=C_{r-j},
\]
and
\[
        C_1<C_2<\cdots<C_{\lfloor r/2\rfloor}.
\]
\end{lemma}

\begin{proof}
The formula for $F_d$ follows from the splitting principle: the Chern roots
of $\operatorname{Sym}^dS^*$ are
\[
        ax+(d-a)y,
        \qquad a=0,1,\ldots,d.
\]
The symmetry $F_d(x,y)=F_d(y,x)$ gives $C_j=C_{r-j}$.

Now put $y=1$.  Since the factors corresponding to $a=0$ and $a=d$ are
$d$ and $dz$, respectively,
\[
        F_d(z,1)=zQ_d(z),
\]
where $Q_d$ has degree $d-1$ and all its roots are real and strictly
negative:
\[
        -\frac{d-a}{a},
        \qquad a=1,\ldots,d-1.
\]
The coefficients of $Q_d$ are exactly $C_1,\ldots,C_d$ and are all
positive; the two omitted coefficients are $C_0=C_r=0$.  Thus Newton's
inequalities (see, for instance, \cite{Brenti1989}) are applied only to the positive support $C_1,\ldots,C_d$.
For $d=2$ the assertion of the lemma is vacuous beyond $C_1>0$.  For
$d\ge3$, the normalized Newton inequalities give, for $2\le j\le d-1$,
\[
 \left(\frac{C_j}{\binom{d-1}{j-1}}\right)^2
 \ge
 \frac{C_{j-1}}{\binom{d-1}{j-2}}
 \frac{C_{j+1}}{\binom{d-1}{j}}.
\]
Since
\[
 \binom{d-1}{j-1}^2
 >
 \binom{d-1}{j-2}\binom{d-1}{j},
\]
and all coefficients are positive, it follows that
\[
        C_j^2>C_{j-1}C_{j+1}.
\]
Hence the successive ratios $C_{j+1}/C_j$ decrease strictly on the positive
support.  The symmetry $C_j=C_{r-j}$ pairs ratios situated symmetrically
about the middle as reciprocals.  Since the ratios decrease strictly, those
before the middle are greater than $1$, and therefore
\[
        C_1<C_2<\cdots<C_{\lfloor r/2\rfloor},
\]
as claimed.
\end{proof}

For a partition $(r-b,b)$ with $0\le b\le\lfloor r/2\rfloor$, the
Schur polynomial in two variables is
\[
        s_{(r-b,b)}(x,y)
        =\sum_{j=b}^{r-b}x^jy^{r-j}.
\]
Write
\[
        F_d(x,y)
        =\sum_{b=0}^{\lfloor r/2\rfloor}
        A_b s_{(r-b,b)}(x,y).
\]
Comparing coefficients for $j\le r/2$ gives
\[
        C_j=A_0+\cdots+A_j,
        \qquad
        A_j=C_j-C_{j-1}.
\]
Let $b_*=\lfloor r/2\rfloor$. By \cref{lem:k2-euler},
\[
        A_{b_*}=C_{b_*}-C_{b_*-1}>0.
\]
The corresponding partition
\[
        \lambda_*=(r-b_*,b_*)
\]
is $(q,q)$ when $r$ is even and $(q,q-1)$ when $r$ is odd.  More is true:
$\lambda_*$ is the unique partition of size $r$ contained in the
$2\times q$ rectangle.  Indeed, if $\lambda=(\lambda_1,\lambda_2)$ has
$|\lambda|=r$ and $\lambda_1\le q$, then
$\lambda_1\ge\lceil r/2\rceil=q$, so necessarily
\[
        \lambda_1=q,
        \qquad
        \lambda_2=r-q.
\]
Thus, in degree $2r$, every Schur class occurring in $F_d$ vanishes in
$H^*(\Gr(2,2+q);\mathbb Z)$ except possibly the class indexed by
$\lambda_*$.  The Schubert-basis description of the Grassmannian cohomology
(see, e.g., \cite[Chapter~9]{FultonYoungTableaux}) gives
\[
        s_{\lambda_*}(S^*)\ne0.
\]
Since its coefficient is $A_{b_*}>0$, we obtain more explicitly
\[
 c_r(\operatorname{Sym}^dS^*)
 =A_{b_*}s_{\lambda_*}(S^*)\ne0,
\]
which proves \eqref{eq:k2-top-chern}.

We have therefore proved directly that every $d$-homogeneous polynomial on
\[
        \CC^{\,2+\lceil(d+1)/2\rceil}
\]
has a two-dimensional zero subspace.  By \cref{lem:incidence-dimension},
the universal property fails in every dimension below the threshold. Thus,
for $d\ge2$,
\[
        h(d,2)=2+\left\lceil\frac{d+1}{2}\right\rceil,
\]
recovering the $k=2$ case of \cref{thm:hom-main} without the general
isotropic-subspace theorem.

The same one-variable mechanism also treats the bounded-degree problem,
but we spell out the reduction because the extra homogeneous components
produce a useful factor that should not be hidden. Set
\[
 R_d:=\frac{d(d+3)}2,
 \qquad
 q_d:=\left\lceil\frac{R_d}{2}\right\rceil,
 \qquad
 m_d:=2+q_d.
\]
Put
\[
 \mathcal E_{\le d}
 =
 \bigoplus_{j=1}^{d}\operatorname{Sym}^j(S^*)
 \longrightarrow \Gr(2,m_d).
\]
A polynomial $P=P_0+P_1+\cdots+P_d$ gives the section
\[
 E\longmapsto
 (P_1|_E,\ldots,P_d|_E),
\]
whose zeros are precisely the two-dimensional subspaces on which $P$ is
constant.  Its rank is
\[
 \operatorname{rank}\mathcal E_{\le d}
 =\sum_{j=1}^d(j+1)=R_d.
\]
By the splitting principle, the top Chern polynomial is
\[
 G_d(x,y)
 =
 \prod_{j=1}^d\prod_{a=0}^j
 \bigl(ax+(j-a)y\bigr).
\]
For each fixed $j$, the two extreme factors, corresponding to $a=0$ and
$a=j$, contribute $j^2xy$.  Hence
\begin{equation}\label{eq:Gd-factorization}
\begin{aligned}
 G_d(x,y)
 &=(d!)^2(xy)^d H_d(x,y),\\
 H_d(x,y)
 &=\prod_{j=2}^d\prod_{a=1}^{j-1}
 \bigl(ax+(j-a)y\bigr).
\end{aligned}
\end{equation}
The polynomial $H_d$ is symmetric and homogeneous of degree
\[
 M_d:=\sum_{j=2}^d(j-1)=\frac{d(d-1)}2.
\]
After setting $y=1$ we obtain
\[
 H_d(z,1)
 =\prod_{j=2}^d\prod_{a=1}^{j-1}(az+j-a),
\]
whose roots are the strictly negative real numbers
\[
        -\frac{j-a}{a},
        \qquad 2\le j\le d,
        \quad 1\le a\le j-1.
\]
Write
\[
        H_d(x,y)=\sum_{\ell=0}^{M_d}D_\ell x^\ell y^{M_d-\ell}.
\]
All $D_\ell$ are positive and, by symmetry,
$D_\ell=D_{M_d-\ell}$. Newton's inequalities, in their normalized form,
give
\[
 \left(\frac{D_\ell}{\binom{M_d}{\ell}}\right)^2
 \ge
 \frac{D_{\ell-1}}{\binom{M_d}{\ell-1}}
 \frac{D_{\ell+1}}{\binom{M_d}{\ell+1}}
 \qquad(1\le\ell\le M_d-1).
\]
Since
\[
 \binom{M_d}{\ell}^2
 >
 \binom{M_d}{\ell-1}\binom{M_d}{\ell+1},
\]
and all coefficients are positive, this implies the strict log-concavity
\[
        D_\ell^2>D_{\ell-1}D_{\ell+1}.
\]
The successive ratios $D_{\ell+1}/D_\ell$ therefore decrease strictly. By
symmetry, ratios situated symmetrically about the middle are reciprocals, so
the ratios before the middle are greater than $1$. Consequently, for $d\ge3$,
the coefficients increase strictly up to the middle:
\[
 D_0<D_1<\cdots<D_{\lfloor M_d/2\rfloor}.
\]
Write
\[
        G_d(x,y)=\sum_{u=0}^{R_d}C_u^{(d)}x^uy^{R_d-u}.
\]
For later reference, the two small degrees can be read off explicitly.
When $d=1$,
\[
        G_1(x,y)=xy,
\]
so $R_1=2$ and the critical coefficient difference is
$C_1^{(1)}-C_0^{(1)}=1$. When
$d=2$,
\[
        G_2(x,y)=4(xy)^2(x+y),
\]
so $R_2=5$ and $C_2^{(2)}-C_1^{(2)}=4$. Thus the critical difference is
positive also in these two cases.

The factorization \eqref{eq:Gd-factorization} shows that
\[
 C_u^{(d)}=
 \begin{cases}
 (d!)^2D_{u-d},& d\le u\le R_d-d,\\
 0,&\text{otherwise}.
 \end{cases}
\]
Set
\[
        b_*:=\left\lfloor\frac{R_d}{2}\right\rfloor.
\]
Since $R_d=M_d+2d$, we have
\[
        b_*-d=\left\lfloor\frac{M_d}{2}\right\rfloor.
\]
Thus, for $d\ge3$, the strict increase of the $D_\ell$ gives
\[
 C_{b_*}^{(d)}-C_{b_*-1}^{(d)}
 =(d!)^2
 \left(
 D_{\lfloor M_d/2\rfloor}
 -D_{\lfloor M_d/2\rfloor-1}
 \right)>0.
\]
For $d=1,2$ the same strict positivity was computed explicitly above.  To
make the Schur-basis step explicit, write
\[
G_d(x,y)
=
\sum_{b=0}^{\lfloor R_d/2\rfloor}
B_b^{(d)}s_{(R_d-b,b)}(x,y).
\]
Comparing coefficients up to the middle gives
\[
        C_u^{(d)}=B_0^{(d)}+\cdots+B_u^{(d)},
        \qquad
        B_u^{(d)}=C_u^{(d)}-C_{u-1}^{(d)}.
\]
Hence the coefficient of
\[
        \Lambda_*=(R_d-b_*,b_*)
\]
is $B_{b_*}^{(d)}=C_{b_*}^{(d)}-C_{b_*-1}^{(d)}$ and is therefore positive.
Moreover, $\Lambda_*$ is the unique partition of size $R_d$ contained in the
$2\times q_d$ rectangle. Hence the top Chern class of $\mathcal E_{\le d}$
is nonzero in
\[
 H^{2R_d}(\Gr(2,m_d);\mathbb Z).
\]
Every section arising from a polynomial of degree at most $d$ must therefore
vanish. Combining this existence statement with
\cref{lem:incidence-dimension} gives
\[
 \mu(d,2)
 =
 2+\left\lceil\frac{R_d}{2}\right\rceil
 =
 2+\left\lceil\frac{d(d+3)}4\right\rceil,
\]
which is the $k=2$ specialization of \cref{thm:general-main}.  Thus the direct
Chern--Newton argument sees both exact thresholds in dimension two without
using either of the two global existence theorems.

\paragraph{Two complementary mechanisms.}
The proofs above separate two features of the problem.  The incidence count
is responsible for the exact numerical threshold and for sharpness below it.
Universal existence is then supplied either by a global projective theorem
(Tevelev in the homogeneous case and Debarre--Manivel, after eliminating the
linear component, for the system of degrees $2,\ldots,d$) or, when $k=2$, by
a direct topological obstruction.
The latter proof shows that the threshold is not merely an artifact of a
genericity theorem: in the first nontrivial rank it is detected by the top
Chern class of the restriction bundle itself.

\section{Consequences and asymptotics}

The homogeneous formula is
\[
h(d,k)=
k+\left\lceil
\frac1k\binom{d+k-1}{k-1}
\right\rceil
\qquad(d\ge3),
\]
while the bounded-degree formula is
\[
\mu(d,k)=
k+\left\lceil
\frac{\binom{d+k}{k}-1}{k}
\right\rceil
\qquad(d\ge3).
\]

The difference is structural. A homogeneous polynomial gives one projective
equation. An arbitrary polynomial of degree at most $d$ gives, after
separating its homogeneous components, a simultaneous system of degrees
\[
        (1,2,\ldots,d).
\]
The two thresholds therefore count different incidence problems.

For fixed $d\ge3$ and $k\to\infty$, one has
\[
\binom{d+k-1}{k-1}
=
\binom{d+k-1}{d}
=
\frac{k^d}{d!}+O(k^{d-1})
\]
and
\[
\binom{d+k}{k}
=
\binom{d+k}{d}
=
\frac{k^d}{d!}+O(k^{d-1}).
\]
Since $d\ge3$, the additive term $k$ is of lower order than
$k^{d-1}$.  Therefore
\[
h(d,k)\sim\frac{k^{d-1}}{d!},
\qquad
\mu(d,k)\sim\frac{k^{d-1}}{d!}.
\]
Although the leading terms agree, Pascal's identity yields
\[
\binom{d+k}{k}-1-\binom{d+k-1}{k-1}
=
\binom{d+k-1}{d-1}-1.
\]
Since each ceiling changes its argument by less than one,
\[
\mu(d,k)-h(d,k)
=
\frac1k\binom{d+k-1}{d-1}+O(1)
\sim
\frac{k^{d-2}}{(d-1)!}.
\]
Thus the bounded-degree condition produces a lower-order, but still
unbounded, increase when the degree is fixed.

The contrast is sharper when $k\ge2$ is fixed and $d\to\infty$.  Indeed,
\[
\binom{d+k-1}{k-1}
=
\frac{d^{k-1}}{(k-1)!}+O(d^{k-2}),
\qquad
\binom{d+k}{k}
=
\frac{d^k}{k!}+O(d^{k-1}),
\]
and hence
\[
h(d,k)\sim\frac{d^{k-1}}{k!},
\qquad
\mu(d,k)\sim\frac{d^k}{k\,k!}.
\]
The bounded-degree threshold then has one higher polynomial order in $d$.

The quadratic identities
\[
        h(2,k)=2k,
        \qquad
        \mu(2,k)=2k+1
\]
are the simplest manifestation of this distinction. The linear term in an
arbitrary quadratic polynomial first forces the constancy subspace into a
hyperplane; the quadratic condition must then be satisfied inside that
hyperplane.

The earlier inductive estimates can be much larger because they construct
zero directions successively, while the Grassmannian treats all
$k$-dimensional subspaces at once. For
example, in the cubic case the exact threshold grows quadratically in $k$,
while the earlier bound displayed in the introduction grows exponentially.

\section*{Author contributions}
All authors contributed to the conceptual development, verification of the arguments,
and writing and revision of the manuscript. All authors read and approved the final
manuscript.

\section*{Funding}
This work was supported by the Conselho Nacional de Desenvolvimento
Cient\'ifico e Tecnol\'ogico (CNPq, Brazil), through Grants No.~406457/2023-9
(CNPq/MCTI Call No.~10/2023) and No.~403964/2024-5 (MCTI/CNPq Call
No.~16/2024). D.~M. Pellegrino was also supported by Grant No.~305807/2025-0,
and A.~Raposo Jr. by Grant No.~302341/2025-0. The funding agency had no role
in the development of the results, the preparation of the manuscript, or the
decision to submit the article.

\section*{Data availability}
No data were used for the research described in this article.

\section*{Declaration of competing interest}
The authors declare that they have no known competing financial interests or
personal relationships that could have appeared to influence the work
reported in this article.

\section*{Declaration of generative AI and AI-assisted technologies in the manuscript preparation process}
During the preparation of this work, the authors used OpenAI's ChatGPT
to assist with mathematical discussion, the exploration of possible
approaches, bibliographic searches, and the improvement of exposition.
After using this tool, the authors reviewed and edited the content as
needed, independently verified all mathematical statements, proofs,
calculations, and references, and take full responsibility for the
content of the publication.

\end{document}